\documentclass[12pt,14paper]{article}
\usepackage[applemac]{inputenc}
\usepackage[cyr]{aeguill}
\usepackage[francais,english]{babel}
\usepackage[scaled=0.92]{helvet}
\usepackage{mathrsfs}
\usepackage{dsfont}
\usepackage{Hfonts}

\usepackage{rotating}

\usepackage{setspace}

\usepackage{xcolor}
\usepackage{graphicx}

\usepackage{pdfsync} 
\usepackage{graphicx}
\usepackage{epstopdf}
\usepackage{color}
\usepackage{caption}
\DeclareCaptionFormat{myformat}{\small{\fontfamily{phv}\selectfont#1#2}#3\par}
\usepackage{amssymb,amsmath}

\usepackage{comment}
\usepackage{enumerate}

\usepackage[linktocpage=true]{hyperref}
\hypersetup{
colorlinks=true, linktocpage=true, pdfstartpage=3, pdfstartview=FitV,
breaklinks=true, pdfpagemode=UseNone, pageanchor=true, pdfpagemode=UseOutlines,
plainpages=false, bookmarksnumbered, bookmarksopen=true, bookmarksopenlevel=1,
hypertexnames=true, pdfhighlight=/O, urlcolor=purple, linkcolor=black, citecolor=green!50!black,
pdftitle={},
pdfauthor={D.D. Baptista de Souza, H. Stein Shiromoto},
pdfsubject={},
pdfkeywords={},
pdfcreator={pdfLaTeX},
pdfproducer={TeXShop}
}

\newcommand{\dx}{\dot{x}}

\newcommand{\class}{\mathcal{C}}

\newcommand{\Kinf}{\mathcal{K}^{\infty}}
\renewcommand{\L}{\mathscr{L}}
\newcommand{\Mi}{\mathscr{M}}
\renewcommand{\P}{\mathscr{P}}

\newcommand{\V}{\mathbf{V}}
\newcommand{\C}{\mathscr{C}}
\newcommand{\D}{\mathscr{D}}
\newcommand{\B}{\mathbf{B}}

\newcommand{\co}{\mathtt{co}\ }

\newcommand{\A}{\mathbf{A}}

\newcommand{\startmodif}{\color{black}}
\newcommand{\stopmodif}{\color{black}}

\newcommand{\on}{\color{black}}
\newcommand{\off}{\color{black}}

\graphicspath{{./imgs/}}

\newtheorem{nntheorem}{\bf Theorem}
\newtheorem{nnassumption}{\bf Assumption}
\newtheorem{nndefinition}{\bf Definition}
\newtheorem{nnlemma}[nndefinition]{\bf Lemma}
\newtheorem{nncorollary}[nndefinition]{\bf Corollary}
\newtheorem{nnproposition}[nndefinition]{\bf Proposition}
\newtheorem{nnexample}[nndefinition]{Example}
\newtheorem{nnproblem}[nndefinition]{Problem}

\newenvironment{theorem}
{\begin{nntheorem}\it}
{\end{nntheorem}}
\newenvironment{proposition}
{\begin{nnproposition}\it}
{\end{nnproposition}}
\newenvironment{lemma}
{\begin{nnlemma}\it}
{\end{nnlemma}}

\newenvironment{assumption}
{\begin{nnassumption}\it}
{\end{nnassumption}}

\newtheorem{nnremark}{Remark}
\newenvironment{remark}{\begin{nnremark}\it}{\end{nnremark}}

\newenvironment{proof}{{\fontfamily{phv}\selectfont Proof.\ }}{\hfill \hspace*{1pt}\hfill $\bullet$}

\title{Relaxed and hybridized backstepping}
\author{Humberto Stein Shiromoto$^{\ddag}$, Vincent Andrieu$^{\ast}$, Christophe Prieur$^\ddag$
\thanks{The authors $\ddag$ are with GIPSA-lab, 38402, Grenoble, France. Work partially supported by HYCON2 Network of Excellence Highly-Complex and Networked Control Systems, grant agreement 257462. The author $\ast$ is with LAGEP, 69622, Lyon, France. E-mail:
{\tt\small  humberto.shiromoto@ieee.org}}
}
\begin{document}

\maketitle

\begin{abstract}
In the present work, we consider nonlinear control systems for which there exist structural obstacles to the design of classical continuous backstepping feedback laws. We conceive feedback laws such that the origin of the closed-loop system is not globally asymptotically stable but a suitable attractor (strictly containing the origin) is practically asymptotically stable. A design method is suggested to build a hybrid feedback law combining a backstepping controller with a locally stabilizing controller. A constructive approach is also suggested employing a differential inclusion representation of the nonlinear dynamics. The results are illustrated for a nonlinear system which, due to its structure, does not have {\it a priori} any globally stabilizing backstepping controller.
\end{abstract}

\section{Introduction}
Over the years, research in control of nonlinear dynamical systems has led to many different tools for the design of (globally) asymptotically stabilizing feedbacks. These techniques require particular structures on the systems. Depending on the assumptions, the designer may use different approaches such as high-gain \cite{GrognardSepulchreBastin1999}, backstepping \cite{FreemanKokotovic2008} or forwarding \cite{MazencPraly96}. However, in the presence of unstructured dynamics, these classical design methods may fail.

For systems \on where \off the classical backstepping technique can not be applied to render the origin globally asymptotically stable, the approach presented in this work may solve the problem by utilizing a combination of a backstepping feedback law \on that renders a suitable compact set globally attractive \off and a locally stabilizing controller. \startmodif\ By assumption, this set\stopmodif\ is contained in the basin of attraction of the system in closed loop with the local controller. The main contribution of this work can be seen as a design of hybrid feedback laws for systems which {\it a priori} do not have a \startmodif controller that globally stabilizes the origin.\stopmodif\ This methodology of hybrid stabilizers is by now \startmodif well known\stopmodif\ \cite{Prieur2001} and it has been also applied for systems that do not satisfy the Brockett's condition (\cite{GoebelPrieurTeel2009} and \cite{HespanhaLiberzonMorse03}). Hybrid feedback laws \on can \off have the advantage of rendering the equilibrium of the closed-loop system robustly asymptotically stable with respect to measurement noise and actuators' errors (\cite{4380508} and \cite{GoebelTeel2006}). We also present a technique to design a continuous local controller satisfying constraints on the basin of attraction of the closed-loop system by using a differential inclusion.

\startmodif To our best knowledge,\stopmodif\ this is the first work suggesting a design method to adapt the backstepping technique to a given local controller in the context of hybrid feedback laws. Other works do exist in the context of continuous controllers (\cite{PrieurAndrieu2010} and \cite{PanEzalKrenerKokotovic01}). In contrast to these works, for the class of systems considered in this paper, \emph{a priori} no continuous globally stabilizing controller does exist. Note that we address a different problem than \cite{MayhewSanfeliceTeel2011CDC}, \startmodif where\stopmodif\ a synergistic Lyapunov function and a feedback law are designed \startmodif by\stopmodif\ backstepping.

In Section \ref{sec:problem statement}, the stabilization problem under consideration is introduced, as well as the concepts on hybrid feedbacks and the set of assumptions. Based on this set of assumptions, a solution is given in Section \ref{sec:results}. In Section \ref{sec:checking the assumptions}, sufficient conditions are developped to verify  \startmodif the\stopmodif\ assumptions by using linear matrix inequalities. An illustration is given in Section \ref{simu}. Proofs of the results are presented in Section \ref{sec:proofs}. Concluding remarks are given in Section \ref{sec:conclusion}.

{\small {\bf Notation}. The boundary of a set $S\subset\R^n$ is denoted by $\partial S$, its \startmodif convex hull\stopmodif\ is denoted by $\co(S)$ \startmodif and its closure by $\bar S$\stopmodif. The identity matrix of order $n$ is denoted $I_n$. The null $m\times n$ matrix is denoted $0_{m\times n}$. For two symmetric matrices, $A$ and $B$, $A>B$ (\startmodif resp. \stopmodif $A<B$) means $A-B>0$ (\startmodif resp. \stopmodif $A-B<0$). Given a continuously differentiable function $f:\startmodif\R^{n-1}\times\R\stopmodif\to\R^n$, $\partial_xf\startmodif(\bar{x},\bar{u})\stopmodif$ \startmodif(resp. $\partial_uf(\bar{x},\bar{u})$)\stopmodif\ stands for the partial derivative of $f$ with respect to $x$ \startmodif(resp. to $u$)\stopmodif\ at $\startmodif(\bar{x},\bar{u})\in\R^{n-1}\times\R\stopmodif$. Let \startmodif$V:\R^{n-1}\to\R$ \stopmodif\ be a continuously differentiable function, $L_f\startmodif V(\bar{x},\bar{u})\stopmodif$ stands for the Lie derivative of \startmodif$V$ in the $f$-direction at $(\bar{x},\bar{u})\in\R^{n-1}\times\R$, i.e., $L_fV(\bar{x},\bar{u}):=\partial_xV(\bar{x})\cdot f(\bar{x},\bar{u})$\stopmodif. Let $c_\ell\in\Rs$ be a constant, we denote sets $\Omega_{c_\ell}(V)=\{x:V(x)<c_\ell\}$. By $r\B_1$ we denote an open ball with radius $r$ and centered at $x_0=0$.}
\section{Problem statement}\label{sec:problem statement}
 
{\itshape\on A. Class of systems.\off} In this work, we consider the class nonlinear systems defined by 
\begin{equation}\label{eq:general system}
\left\{\begin{array}{rcl}
\dx_1&=&f_1(x_1,x_2)+h_1(x_1,x_2,u)\\
\dx_2&=&f_2(x_1,x_2)u+h_2(x_1,x_2,u),
\end{array}\right.
\end{equation}
where $(x_1,x_2)\in\R^{n-1}\times\R$ is the state and $u\in\R$ is the input. Functions $f_1$, $f_2$, $h_1$ and $h_2$ are locally Lipschitz continuous satisfying $f_1(0,0)=0$, $h_1(0,0,0)=0$, \startmodif$h_2(0,0,0)=0$ \stopmodif and, $\forall(x_1,x_2)\in\R^{n-1}\times\R$, $f_2(x_1,x_2)\neq0$. \startmodif In\stopmodif\ a more compact notation, the vector $(x_1,x_2)\in\R^{n-1}\times\R$ is denoted by $x$, the $i$-th component of $x_1\in\R^{n-1}$ is denoted by $x_{1,i}$ and system \eqref{eq:general system} is denoted by $\dx=f_h(x,u)$. When $h_1(x_1,x_2,u)\equiv0$ and $h_2(x_1,x_2,u)\equiv0$, system \eqref{eq:general system} is denoted by $\dx=f(x,u)$.

\startmodif Consider\stopmodif\ system $\dx=f(x,u)$, \startmodif assuming\stopmodif\ stabilizability of the $x_1$-subsystem  \startmodif and applying backstepping\stopmodif, one may\startmodif\ design a feedback law $\varphi_b$ by solving an algebraic equation in the input variable $u$. The solution of this algebraic equation equation renders\stopmodif\ the origin  globally asymptotically stable for $\dx=f(x,\varphi_b\startmodif(x)\stopmodif)$ (e.g. \cite{KrsticKokotovicKanellakopoulos95}). However, \startmodif because functions $h_1$ and $h_2$ depend on $u$, the design of a feedback law for \eqref{eq:general system} \stopmodif leads to an implicit equation \startmodif in this variable (see \startmodif\eqref{eq:example:backstepping} below for an example). Thus,  \stopmodif backstepping may be difficult (if not impossible) to apply for \eqref{eq:general system}.

This \startmodif is the motivation to\stopmodif\ introduce the\startmodif\ hybrid feedback law design problem ensuring\stopmodif\ global asymptotic\startmodif\ stability\stopmodif\ of the origin for \eqref{eq:general system} in closed loop.

{\itshape\startmodif B. Preliminaries.\stopmodif} In this section, we give a brief introduction on {\em hybrid feedback laws} $\mathds{K}$ (see \cite{4380508} for further details).  It consists of a \startmodif finite\stopmodif\ discrete set $Q$, \startmodif closed sets $\mathscr{C}_q,\mathscr{D}_q\subset \R^n$ such that, $\forall q\in Q$, $\mathscr{C}_q\cup\mathscr{D}_q=\R^n$,\stopmodif\ continuous functions $\varphi_q:\R^n\to \R$ and \startmodif outer semicontinuous set-valued functions $g_q:\R^n\rightrightarrows Q$\stopmodif. The system \eqref{eq:general system} in closed loop with $\mathds{K}$ leads to a system with mixed dynamics\footnote{Continuous and discrete dynamics.}
\begin{equation}\label{eq:hybrid system}
 \left\{\begin{array}{rcll}
  \dot{x}&=&f_h(x,\varphi_q(x)),&x\in\mathscr{C}_q,\\
  q^+&\in&g_q(x),&x\in\mathscr{D}_q,
 \end{array}\right.
\end{equation}
with state space given by $\R^n\times Q$.

\startmodif In this work, we consider the framework provided in \cite{Goebeletal2012}. A set $S\subset\Ras\times\mathbb{N}$ is called a \emph{\on compact hybrid \off \startmodif time domain \stopmodif} (\cite[Definition 2.3]{Goebeletal2012}) if \on $S=\cup_{j=0}^{J-1}([t_j,t_{j+1}],j)$ for some finite sequence of times $0=t_0\leq t_1\leq\cdots\leq t_J$. It is a hybrid time domain if, $\forall(T,J)\in S$, $S\cap([0,T]\times\{0,\ldots,J\})$ is a compact hybrid time domain\off. A solution of \eqref{eq:hybrid system} is a pair of functions $(x,q):S\times S\to\R^n\times Q$ such that, during flows, $x$ evolves according to the differential equation $\dx=f_h(x,\varphi_q(x))$, while $q$ remains constant, and the constraint $x\in\C_q$ is satisfied. During jumps, $q$ evolves according to the difference inclusion $q^+\in g_q(x)$, while $x$ remains constant, and before a jump the constraint $x\in\D_q$ is satisfied.\stopmodif

Since, $\forall q\in Q$, each function $\varphi_q$ is continuous and each \startmodif set-valued \stopmodif function $g_q$ is outer semicontinuous \startmodif and locally bounded\footnote{\startmodif Because $\hat{Q}$ is finite.\stopmodif}\stopmodif, system \eqref{eq:hybrid system} satisfies the basic assumptions of \cite{4806347}.\footnote{\on The interested reader might also consult \cite{4380508} and \startmodif\cite{Goebeletal2012} \stopmodif \on for the sufficient conditions for the existence of solutions and  stability (of the origin or subsets of $\R^n$) concepts.\off} 

\subsection*{C. Assumptions}

\startmodif
\begin{assumption}{(Local hybrid controller)}\label{hyp:local stability} 

	Let $\hat{Q}\subset\mathbb{N}$ be a finite discrete set\on, f\off\startmodif or each $\hat{q}\in \hat{Q}$,

	${\on\bullet\off}$ the sets $\mathscr{C}^\ell_{\hat{q}}\subset\R^n$ and $\mathscr{D}^\ell_{\hat{q}}\subset\R^n$ are closed, and $\mathscr{C}_{\hat{q}}^\ell\cup\mathscr{D}_{\hat{q}}^\ell=\R^n$;
	
	${\on\bullet\off}$ $\varphi_{\hat{q}}^\ell:\R^n\to\R$ is a continuous function and $g_{\hat{q}}^\ell:\R^n\rightrightarrows\hat{Q}$ is an outer semicontinuous set-valued function; 
	
	${\on\bullet\off}$ $V_{\hat{q}}^\ell:\R^n\to\Ras$ is a $\mathcal{C}^1$ function satisfying, $\forall x\in\R^n$, $\alpha_1(||x||)\leq V^\ell_{\hat{q}}(x)\leq \alpha_2(||x||)$, where $\alpha_1,\alpha_2\in\mathcal{K}_\infty$. Moreover, let $c_{\ell}>0$ be a constant, it also satisfies, $\forall\hat{q}\in\hat{Q}$,
		\begin{eqnarray}
			\forall x\in(\Omega_{c_\ell}(V_{\hat{q}}^\ell)\cap\mathscr{C}_{\hat{q}})\setminus\{0\},&L_{f_{h}}V_{\hat{q}}^\ell(x,\varphi_{\hat{q}}^{\ell}(x))<0,\label{eq:hypothesis:1}\\
			\forall x\in(\Omega_{c_\ell}(V_{\hat{q}}^\ell)\cap\mathscr{D}_{\hat{q}})\setminus\{0\},&V_{\hat{q}^+}^\ell(x)-V_{\hat{q}}^\ell(x)<0.\label{eq:hypothesis:1:discrete}
		\end{eqnarray}
\end{assumption}

System \eqref{eq:general system} in closed loop with the local hybrid controller leads to the hybrid system
	\begin{equation}\label{eq:local hybrid system}
 		\left\{\begin{array}{rcll}
 			\dot{x}&=&f_h(x,\varphi^\ell_{\hat{q}}(x)),&x\in\mathscr{C}_{\hat{q}}^\ell,\\
  			\hat{q}^+&\in&g^\ell_{\hat{q}}(x),&x\in\mathscr{D}_{\hat{q}}^\ell.
			 \end{array}\right.
	\end{equation}
	Due to \cite[Theorem 20]{4806347}, Assumption \ref{hyp:local stability} implies that the set $\{0\}\times\hat{Q}$ (which will be called origin) is locally asymptotically stable for \eqref{eq:local hybrid system}. Whenever we are in a neighborhood of the origin, Equation \eqref{eq:hypothesis:1} implies that the Lyapunov function $\R^n\times\hat{Q}\ni(x,\hat{q})\mapsto V_{\hat{q}}(x)\in\Ras$ is strictly decreasing during a flow. Equation \eqref{eq:hypothesis:1:discrete} implies that, during a transition from a controller $\hat{q}$ to a controller $\hat{q}^+$, the value $V_{\hat{q}}(x)$ strictly decreases to $V_{\hat{q}^+}(x)$.
\stopmodif

 The second assumption provides bounds for the terms that impeach the direct application of the backstepping method. \startmodif It also\stopmodif\ concerns the global stabilizability of the origin for
\begin{equation}\label{eq:general subsystem:x1 h0}
\dx_1=f_1(x_1,x_2),
\end{equation}
when $x_2$ is considered as an input.

\begin{assumption}{(Bounds)}\label{hyp:bounds}
There exist a $\class^1$ positive definite and proper \startmodif function\stopmodif\ $V_1:\R^{n-1}\to\Ras$, a $\class^1$ function $\psi_1:\R^{n-1}\to\R$, a $\Kinf$ and locally Lipschitz function $\alpha:\Ras\to\Ras$ such that

a) (Stabilizing feedback to \eqref{eq:general subsystem:x1 h0}): $\forall x_1\in\R^{n-1}$, $L_{f_1}V_1(x_1,\psi_1(x_1))\leq-\alpha(V_1(x_1))$;\\
In addition, there exist a continuous function $\Psi:\R^n\to\R$ and two positive constants, $\varepsilon<1$ and $M$, satisfying

b) (Bound on $h_1$): $\startmodif\forall (x_1,x_2,u)\in\R^{n-1}\times\R\times\R$, $||h_1\startmodif(x_1,x_2,u)\stopmodif||\leq\Psi\startmodif(x_1,x_2)\stopmodif$ and $L_{h_1}V_1(x_1,\psi_1(x_1),u)\leq(1-\varepsilon)\alpha(V_1(x_1))+\varepsilon\alpha(M)$;

c) (Bound on $\partial_{x_2}h_1$): $\startmodif\forall (x_1,x_2,u)\in\R^{n-1}\times\R\times\R\stopmodif$, $\left|\left|\partial_{x_2} h_1\startmodif(x_1,x_2,u)\stopmodif\right|\right|\leq\Psi\startmodif(x_1,x_2)\stopmodif$;

d) (Bound on $h_2$): $\startmodif\forall (x_1,x_2,u)\in\R^{n-1}\times\R\times\R\stopmodif$, $||h_2\startmodif(x_1,x_2,u)\stopmodif||\leq\Psi\startmodif(x_1,x_2)\stopmodif$.
\end{assumption}

Before introducing the last assumption, let the set $\A\subset \R^n$ be given by
\begin{equation}\label{eq:general set:A}
\A=\{(x_1,x_2)\in\R^n:V_1(x_1)\leq M,x_2=\psi_1(x_1)\}.
\end{equation}
Since function $V_1$, given by Assumption \ref{hyp:bounds}, is proper, this set is compact. Moreover, we will prove in Proposition \ref{prop:global backstepping} that if Assumption \ref{hyp:bounds} holds, then there exists a feedback law $\varphi_g$ rendering $\A$ globally practically stable for $\dx=f_h(x,\varphi_g(x))$.

\begin{assumption}{(Inclusion)}\label{hyp:inclusion}
\startmodif For each $\hat{q}\in\hat{Q}$, each function $V_{\hat{q}}^\ell$ satisfies $$\displaystyle\max_{x\in\A}V_{\hat{q}}^\ell(x)<c_\ell.$$\stopmodif
\end{assumption}
The first and third assumptions together ensure that\startmodif,  for each $\hat{q}\in\hat{Q}$, $\A$\stopmodif\  is included in the basin of attraction of system \eqref{eq:local hybrid system}. In the following section, it will be shown that the above assumptions are sufficient to solve the problem under consideration.

\section{Results}\label{sec:results}

\startmodif Before stating the first result, we recall the concept of global practical asymptotical stability. A compact set $\bS\subset\R^n$ \on containing the origin \off is \emph{globally practically asymptotically stabilizable} for \eqref{eq:general system} if, $\forall a\in\Rs$, there exists a controller $\varphi_g$ such the set $\bS + a\B_1$ contains a set \on that \off is globally asymptotically stable for $\dx=f_h(x,\varphi_g(x))$ \on (\cite{Teel:1999})\off.
\stopmodif

\begin{proposition}\label{prop:global backstepping}
Under Assumption \ref{hyp:bounds}, $\A$ is globally practically \startmodif stabilizable for \eqref{eq:general system}.
\end{proposition}

\begin{comment}
The proof of Proposition \ref{prop:global backstepping} provides a feedback law of the following structure
\begin{equation}\label{eq:global backstepping}
\begin{split}
\varphi_g(x_1,x_2)=\frac{1}{K_V f_2(x_1,x_2)}\cdot[K_VL_{f_1}\psi_1(x_1,x_2)-\partial_{x_1}V_1(x_1)\\
\cdot\int_0^1\partial_{x_2}f_1(x_1,\eta_{x_1,x_2}(s))\,ds-(x_2-\psi_1(x_1))\cdot&(c+\frac{c}{4}\Delta^2(x_1,x_2))],
\end{split}
\end{equation}
where $K_V$ and $\Delta$ are, respectively, a suitable positive constant and a function (defined respectively by \eqref{def:k}  and \eqref{eq:Delta} below).
\end{comment}

\on The proof of Proposition \ref{prop:global backstepping} is provided in Section \ref{proof:prop:global backstepping}. \off We can now state the main result.

\begin{theorem}\label{thm:hybrid feedback}\startmodif
Under Assumptions \ref{hyp:local stability}, \ref{hyp:bounds} and \ref{hyp:inclusion}, there exists a continuous controller $\varphi_g:\R^n\to\R$; a suitable choice of a constant value $\tilde c_\ell$ satisfying $0<\tilde{c}_\ell<c_\ell$; a hybrid state feedback law $\mathds{K}$ defined by $Q:=\{1,2\}\times \hat{Q}$ and, $\forall \hat{q}\in\hat{Q}$, subsets of $\R^n$
\normalsize
\begin{equation}\label{def:C}
\begin{array}{rcl}
\mathscr{C}_{1,\hat{q}}=\overline{\Omega_{c_\ell}(V_{\hat{q}}^\ell)}\cap\mathscr{C}_{\hat{q}}^\ell,
&\mathscr{C}_{2,\hat{q}}=\overline{\R^n\setminus\Omega_{\tilde{c}_\ell}(V_{\hat{q}}^\ell)},\\ \\\multicolumn{3}{c}{\mathscr{D}_{2,\hat{q}}=\overline{\Omega_{\tilde{c}_\ell}(V_{\hat{q}}^\ell)},}\\ \\
\multicolumn{3}{l}{\mathscr{D}_{1,\hat{q}}=(\overline{\Omega_{c_\ell}(V_{\hat{q}}^\ell)}\cap\mathscr{D}_{\hat{q}}^\ell)\cup(\overline{\R^n\setminus\Omega_{{c}_\ell}(V_{\hat{q}}^\ell)})}
\end{array}
\end{equation}
and functions
\begin{equation}\label{def:phi1}
\begin{array}{rcl}
	\multicolumn{3}{l}{
	\begin{array}{rcl|rcl}
	\varphi_{1,\hat{q}}:\mathscr{C}_{1,\hat{q}}&\to&\R&\varphi_{2,\hat{q}}:\mathscr{C}_{2,\hat{q}}&\to&\R\\
	x&\mapsto&\varphi_{\hat{q}}^\ell(x)
	&x&\mapsto&\varphi_g(x)\end{array}}\\
	\hline\multicolumn{3}{c}{\begin{array}{rcl} 
	 g_{2,\hat{q}}:\mathscr{D}_{2,\hat{q}}&\rightrightarrows&Q\\
	x&\mapsto&\{(1,\hat{q})\}\end{array}}\\
	\hline g_{1,\hat{q}}:\mathscr{D}_{1,\hat{q}}&\rightrightarrows&Q\\
	\multicolumn{3}{c}{x\mapsto\left\{\begin{array}{rl}
									\{(1,g_{\hat{q}}^\ell(x))\},&x\in(\Omega_{c_\ell}(V_{\hat{q}}^\ell)\cap\mathscr{D}_{\hat{q}}^\ell)\\
									\{(2,\hat{q})\},&x\in(\R^n\setminus\overline{\Omega_{{c}_\ell}(V_{\hat{q}}^\ell)})\\
									\{(1,g_{\hat{q}}^\ell(x)),(2,\hat{q})\},&x\in(\partial\Omega_{c_\ell}(V_{\hat{q}}^\ell)\cap\mathscr{D}_{\hat{q}}^\ell)
								 \end{array}\right.}
\end{array}
\end{equation}
such that the origin is globally asymptotically stable for system \eqref{eq:general system} in closed loop with $\mathds{K}$.
\end{theorem}
\stopmodif

 Theorem \ref{thm:hybrid feedback} is more than an existence result since its proof allows one to design a hybrid feedback law \startmodif that globally stabilizes the origin\stopmodif. The complete proof of Theorem 1 is presented in Section \ref{sec:proof:thm 1}.
\begin{comment}
Let us sketch the proof of Theorem \ref{thm:hybrid feedback}. First, we use Assumption \ref{hyp:bounds} to apply Proposition \ref{prop:global backstepping} in order to design a controller, denoted $\varphi_g$, such that the set $\A$ is globally practically stable for $\dx=f_h(x,\varphi_g)$. Using Assumptions \ref{hyp:local stability} and \ref{hyp:inclusion}, this set is shown to be included in the basin of attraction of $\dx=f_h(x,\varphi_\ell)$. Then we design a hybrid feedback law based on an hysteresis of both controllers $\varphi_\ell$ and $\varphi_g$ on appropriate sets. This latter construction is adapted from other works like \cite{4806347} and \cite{Prieur2001}. 
\end{comment}

\begin{remark}
\startmodif  These results are useful for systems that do not satisfy the Brockett necessary condition for the existence of a continuous stabilizing controller and for which there exists a locally stabilizing hybrid feedback law (see e.g. \cite[Example 38]{4806347} and \cite{HespanhaMorse99}).
\stopmodif
\end{remark}

The problems concerning the design of a feedback law by backstepping to \eqref{eq:general system} that renders $\A$ globally practically stable and blends different feedback laws according to each basin of attraction is solved. In the next section, under Assumption \ref{hyp:bounds},  a local continuous feedback law satisfying Assumptions \ref{hyp:local stability} and \ref{hyp:inclusion} is designed using an over-approximation of \eqref{eq:general system}.

\section{A method to design a local stabilizing controller}\label{sec:checking the assumptions}

Based on the approach presented on \cite{AndrieuTarbouriech2011}, we formulate the nonlinear dynamics of \eqref{eq:general system} in terms of a Linear Differential Inclusion. \startmodif Let $\hat{Q}$ be a singleton, we start by defining a neighborhood $\V_\mu$ of the origin such that: a) there exist a continuous feedback law $\varphi_{\hat{q}}^\ell$ and a Lyapunov function $V_{\hat{q}}^\ell$ satisfying, $\forall x\in\V_\mu\setminus\{0\}$, $L_{f_h}V_{\hat{q}}^\ell(x,\varphi_{\hat{q}}^\ell(x))<0$; b) it strictly contains an estimation of the basin of attraction of $\dx=f_h(x,\varphi_{\hat{q}}^\ell(x))$ and a convex set containing $\A$. These two set inclusions follow from two over-approximations of $\mathbf{A}$.\stopmodif

\startmodif Under \stopmodif Assumption \ref{hyp:bounds}, \startmodif there exist\stopmodif\footnote{\startmodif Because $\A$ is a compact set.\stopmodif} a finite set $\P\subset\mathbb{N}$ of indexes and $\{x_p\}_{p\in\P}$ vectors of $\R^n$ such that
\begin{equation}\label{eq:A included in a convex}
\A\subset\co(\{x_p\}_{p\in\P}).
\end{equation}

\begin{comment}

Let us introduce the vectors $A_p\in\R^{n-1\times 1}$, $p=1,\ldots,n-1$, with components given by either
\begin{equation}
\begin{array}{rcl}
a_{j}^+=\displaystyle\max_{x_1:V_1(x_1)\leq M} \partial_{x_{1,j}} \psi_1(x_1)&\text{or}& a_{j}^-=\displaystyle\min_{x_1:V_1(x_1)\leq M} \partial_{x_{1,j}} \psi_1(x_1).
\end{array} 
\end{equation}
Indeed one way to construct vectors $x_p$ and set $\P$ is by using $\psi_1$ together with the mean value theorem. More precisely, for each $x_1$ such that $V_1(x_1)\leq M$ there exists a pair of vectors $A_i$, $A_j$ such that $A_i\cdot x_1\leq \psi_1(x_1)\leq A_j\cdot x_1$. Then, we define the set $\{x_p\}_{p\in\P}$, $x_p=(x_1,x_2)\in\R^n$ by $(\{x_1:V_1(x_1)=M\}\times \{A_ix_1\})\cup(\{x_1:V_1(x_1)=M\}\times \{A_jx_1\})$.
\end{comment}

Let $\mu_u>0$ be a constant and $\mu=[\mu_1,\mu_2,\ldots,\mu_n]\in\R^{n}$ be a vector of positive values such that $\co(\{x_p\}_{p\in\P})\subset\V_\mu=\{x:|x_i|\leq\mu_i,i=1,2,\ldots,n\}$. 

Consider the function
\begin{equation}\label{eq:general system:difference}
\tilde{f}_h(x,u)=f_h(x,u)-Fx-Gu,
\end{equation}
where $F$ and $G$ are\startmodif\ the linearization of \eqref{eq:general system} around the origin:
\begin{equation}\label{eq:general system:origin}
\dx=Fx+Gu:=\partial_x f_h(0)x+\partial_uf_h(0)u.
\end{equation}
Since $f_h$ is $\class^1$, $\tilde{f}_h$ is also $\class^1$. \on In the following, an elementwise over-approximation is made of the matrices in \eqref{eq:general system:origin}. \off \stopmodif For each $l\in\mathscr{L}:=\{l\in\mathbb{N}:1\leq l\leq2^{n^2}\}$, let $C_l\in\R^{n\times n}$ be a matrix with components given by either
\begin{equation}\label{eq:matrix C:elements}
\begin{split}
c_{ij}^+=\displaystyle\max_{x\in \V_\mu, |u|\leq\mu_u} \partial_{x_j} \tilde{f}_{h,i}(x,u)\ \text{or}\\
c_{ij}^-=\displaystyle\min_{x\in \V_\mu, |u|\leq\mu_u}&\partial_{x_j }\tilde{f}_{h,i}(x,u).
\end{split}
\end{equation}
For each $m\in\mathscr{M}:=\{m\in\mathbb{N}:1\leq m\leq2^n\}$, let $D_m\in\R^{n\times 1}$ be a vector with components given by either
\begin{equation}\label{eq:matrix D:elements}
\begin{split}
d_i^+=\displaystyle\max_{x\in \V_\mu, |u|\leq\mu_u}  \partial_u \tilde{f}_{h,i}(x,u)\ \text{or}\\
d_i^-=\displaystyle\min_{x\in \V_\mu, |u|\leq\mu_u}&  \partial_u \tilde{f}_{h,i}(x,u).
\end{split} 
\end{equation}

For each $i\in\mathscr{I}:=\{i\in\mathbb{N}:1\leq i\leq n\}$, the mean value theorem ensures, for all $x\in\V_\mu$ and $|u|\leq\mu_u$, the existence of points $\overline{x}$ and $\overline{u}$ satisfying $\tilde{f}_{h,i}(x,u)=\partial_x \tilde{f}_{h,i}(\overline{x},\overline{u})x_i+\partial_u\tilde{f}_{h,i}(\overline{x},\overline{u})u$. This implies that, for all $x\in\V_\mu$, $|u|\leq\mu_u$, $l\in\L$ and $m\in\Mi$, \eqref{eq:general system} may be over-approximated by
\begin{equation}\label{eq:general LDI}
\dx\in\co\{(F+C_l)x+(G+D_m)u\}.
\end{equation}

\begin{remark}\startmodif
	This linear differential inclusion allows us to go further than the linearization \eqref{eq:general system:origin} because we take into account the  gradient of the nonlinear terms. The precision of this over-approximation method depends basically on two aspects: the size of the neighborhood $\mathbf{V}_\mu$ considered and the rate of change of the nonlinear terms $\tilde{f}_h$.
	\end{remark}\stopmodif

% \begin{remark} The dimension of the state space has an important role in the cardinality of $\L$, $\Mi$, and $\P$ which will not always be equal to $2^{n^2}$, $2^n$,  and $2^{n}$, respectively, since they also depend on the structure of each function $\tilde{f}_{h}$ and $\psi_1$.
%\end{remark}

%The strategy behind such a formulation is to define a set $\V_\mu$, to design a feedback law $\varphi_\ell$ and a candidate Lyapunov function $V_\ell$ such that  $L_{f_h}V_\ell(x,\varphi_\ell(x))<0$, $\forall x\in\V_\mu$, $x\neq0$; $|\varphi_\ell(x)|\leq\mu_u$, $\forall x\in\overline{\Omega}_1(V_\ell)$; and $\overline{\Omega}_1(V_\ell)\subset\V_\mu$, hence Assumption \ref{hyp:local stability} would hold. Furthermore, by ensuring that $\A$ is contained in $\overline{\Omega}_1(V_\ell)$, Assumption \ref{hyp:inclusion} would be satisfied. 
Let us consider the canonical basis in $\R^n$, i.e., the set of vectors $\{e_s\}_{s=1}^n$, where the components are all 0 except the $s$-th one which is equals to 1.
\begin{proposition}\label{prop:local feedback}
Assume that there exist a symmetric positive definite matrix $W\in\R^{n\times n}$ and a matrix $H\in\R^{n\times 1}$ satisfying, for all $l\in\L$ and $m\in\Mi$, 
\begin{equation}\label{eq:LMI:a}
\begin{split}
W(F+C_l)^T+H(G+D_m)^T+(F+C_l)W\\
+(G+D_m)&H^T<0,
\end{split}
\end{equation}%\addtocounter{equation}{1}
\begin{equation}\label{eq:condition:derivee de Lyap}%\tag{\theequation-left}\label{eq:condition:derivee de Lyap}
\begin{bmatrix}
\mu_s^2W&We_s\\
\ast&1
\end{bmatrix}\geq0, s=1,2,\ldots,n,
\end{equation}
\begin{equation}\label{eq:condition:coA}%\tag{\theequation-right}\label{eq:condition:coA}
\begin{bmatrix}
1&x_p^T\\
\ast&W
\end{bmatrix}\geq0,\quad p\in\P,
\end{equation}
and
\begin{equation}\label{eq:condition:K}
\begin{bmatrix}
\mu_u^2W&H\\
\ast&1
\end{bmatrix}\geq0.
\end{equation}
Then, \startmodif by letting $\hat{Q}=\{1\}$, $V_1^\ell(x)=x^TPx$, where $P=W^{-1}$, $c_\ell=1$, $\mathscr{C}_1^\ell=\overline{\Omega_1(V_1^\ell)}$, $\mathscr{D}_1^\ell=\overline{\R^n\setminus\mathscr{C}_1^\ell}$, $g_1^\ell(x)\equiv1$ and $\varphi_1^\ell(x)=Kx$, where $K=H^TP$, Assumptions \ref{hyp:local stability} and \ref{hyp:inclusion} hold.\stopmodif
\end{proposition}

The proof of Proposition \ref{prop:local feedback} \on is provided in \off Section \ref{sec:proof:local feedback}.

\section{Illustration}\label{simu}

Let us consider a class of systems given by
\begin{equation}\label{eq:example:1}
  \left\{\begin{array}{rcl}
   \dot{x}_1&=&x_1+x_2+\theta [x_1^2+(1+x_1)\sin (u)]\\
   \dot{x}_2&=&u,
  \end{array}\right.
 \end{equation}
where $\theta\in\Rs$ is a constant. We will show in the following that, due to the presence of the term $\theta (1+x_1)\sin(u)$ in the time-derivative of $x_1$, it is not possible to apply the backstepping technique to design a feedback law. Let $f_1(x_1,x_2)=x_1+x_2+\theta x_1^2$, $f_2(x_1,x_2)\equiv1$, $h_1(x_1,x_2,u)=\theta(1+x_1)\sin(u)$ and $h_2(x_1,x_2,u)\equiv0$. Before applying Proposition \ref{prop:global backstepping} and Theorem \ref{thm:hybrid feedback}, we check their assumptions.

\subsection{Checking assumptions for \eqref{eq:example:1}}

\subsubsection{Assumption \ref{hyp:bounds}} 
It is possible to check that item a) holds with $V_1(x_1)=x_1^2/2$, $\psi_1(x_1)=-(1+K_1)x_1-\theta x_1^2$, and $\alpha(s)=2K_1s$, where $K_1\in\Rs$ is constant. Items b)-d) hold with $\Psi(x_1,x_2)=\theta(1+|x_1|)$, $\varepsilon\leq 1-3\theta/(2K_1)$ and\footnote{\startmodif With $\varepsilon\leq 1-3\theta/(2K_1)$ and the condition $\varepsilon>0$, we get the lower bound for $K_{1}>\frac{3\theta}{2}$.\stopmodif} $M\geq \theta/(4K_1\varepsilon)$.

Since Assumption \ref{hyp:bounds}a) holds, we could try to apply the backstepping technique directly in \eqref{eq:example:1}. Following this procedure, we intend to use $V(x_1,x_2)=V_1(x_1)+(x_2-\psi_1(x_1))^2/2$ as a control Lyapunov function for \eqref{eq:example:1}. Taking its Lie derivative, algebraic computations yield, $\forall(x_1,x_2,u)\in\R^{n-1}\times\R\times\R$,
\begin{equation}\label{eq:example:backstepping}
\begin{array}{l}
\startmodif L_{f_h}V(x_1,x_2,u)\leq\startmodif-K_{1}x_{1}^{2} +x_{1}\theta(1+x_{1})\cdot\sin(u)\\
			\startmodif +(x_{2}-\psi_{1}(x_{1}))(u+x_{1}/2
			\startmodif+(1+K_{1}+2\theta K_{1}x_1)\\
			\startmodif \cdot(x_{1}+x_{2}+\theta[x_{1}^{2}+(1+x_{1})
			\startmodif\cdot\sin(u)])).\stopmodif
\end{array}
\end{equation}
In order to have a term proportional to $(x_2-\psi_1(x_1))^2$ in the right-hand side of \eqref{eq:example:backstepping}, we should solve an implicit equation in the variable $u$ defined as $E(x_1,x_2,u)=-K_1x_1^2+L(x_2-\psi_1(x_1))^2$, where $E$ is the right-hand side of \eqref{eq:example:backstepping} and $L>0$ is a constant. Since this seems to be difficult (if not impossible), it motivates the design a hybrid feedback by applying Theorem \ref{thm:hybrid feedback}.

\subsubsection{Assumptions \ref{hyp:local stability} and \ref{hyp:inclusion}}
From the previous definitions of $V_1$ and $\psi_1$, we get $\A=\{(x_1,x_2):|x_1|<\sqrt{2M},x_2=-(1+K_1)x_1-\theta x_1^2\}.$

Let us first establish sets $\P$ and $\{x_p\}_{p\in\P}$ such that \eqref{eq:A included in a convex} holds. Define constants  $a^+=\max_{|x_1|<\sqrt{2M}} -(1+K_1)- 2\theta x_1=-(1+K_1)+ 2\theta\sqrt{2M}$ and $a^-=\min_{|x_1|<\sqrt{2M}} -(1+K_1)- 2\theta x_1=-(1+K_1)- 2\theta\sqrt{2M}$ by computing the derivatives of $\psi_1$ and let $\P=\{1,2,3,4\}$. From the mean value theorem we have, $\forall x\in\A$, $a^-\cdot x_1\leq x_2\leq a^+\cdot x_1$. This implies that
\begin{equation}\label{eq:example:A included in a convex}
\begin{split}
\A\subseteq\co((\{\sqrt{2M}\}\times\{x_{2}^{+,<0},x_{2}^{-,<0}\})\\
\cup(\{-\sqrt{2M}\}\times\{x_{2}^{+,>0},&x_{2}^{-,>0}\})),
\end{split}
\end{equation}
where
\begin{equation}\label{eq:vertices of coA}
\begin{array}{rl}
x_{2}^{+,>0}=-a^+\sqrt{2M},&x_{2}^{+,<0}=a^+\sqrt{2M},\\
x_{2}^{-,>0}=-a^-\sqrt{2M},&x_{2}^{-,<0}=a^-\sqrt{2M}. 
\end{array}
\end{equation}
Figure \ref{fig:coA:0.1} shows this inclusion.
\begin{figure}[ht!]
 \centering
 \includegraphics[width=6.5cm]{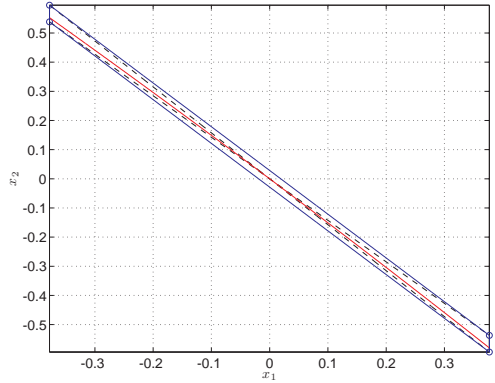}
 \caption{The sets $\A$ (in red) and the convex set defined in \eqref{eq:example:A included in a convex} (in blue) are presented in solid line. The circles are the vertexes of this set. The dashed straight lines which bound $\A$ are given by functions $x_1\mapsto a^+x_1$ and $x_1\mapsto a^-x_1$.}
 \label{fig:coA:0.1}
\end{figure}

\begin{figure}[ht!]
 \centering
 \includegraphics[width=7cm]{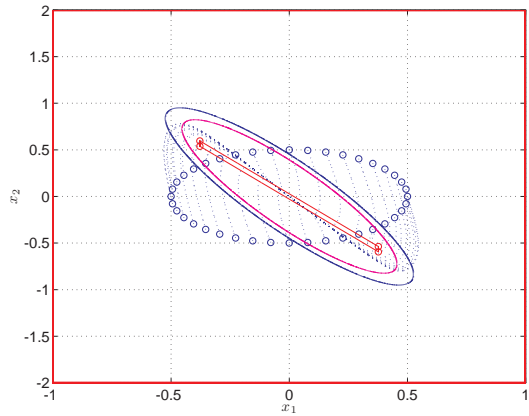}
\caption{The sets $\V_\mu$ (in red), $\partial\Omega_1(x^TPx)$ (in blue), and the inclusion \eqref{eq:example:A included in a convex} (in red) at the center. Initial conditions are points given in a ball of radius 0.5 and centered at the origin.}
	\label{fig:coA:0.1:b}
\end{figure}

A necessary condition for feasibility of the Linear Matrix Inequalites of Proposition \ref{prop:local feedback} is $\A\subset\V_\mu$. This follows from the inequalities $\sqrt{2M}<\mu_1$ and $|a^\pm\sqrt{2M}|<\mu_2$. These inequalities imply that $K_1$ must satisfy 
\begin{equation}\label{eq:feasability:K1}
\tfrac{\theta}{2}\left(\textstyle\frac{1}{\mu_1^{2}}+3\right)<K_1<\textstyle\frac{\mu_2}{\mu_1}-2\theta\mu_1-1.
\end{equation}
\startmodif
\begin{remark}
	Equation \eqref{eq:feasability:K1} imposes a limitation on the speed of response, since $K_1$ is lower and upper bounded.
\end{remark}
\stopmodif

Let $\theta=0.1$, applying the technique presented in Section \ref{sec:checking the assumptions} we define $\mu=[1,2]$, $\V_\mu=\{(x_1,x_2):|x_1|<1,|x_2|<2\}$ and $|u|<2\pi$. Moreover, letting $K_1=0.5$ we get that \eqref{eq:feasability:K1} holds. From Assumption \ref{hyp:bounds} and with this choice for $K_1$, we let $M=0.0714$, and $\varepsilon=0.7$.

The matrices $F$ and $G$ defined in \eqref{eq:general system:origin} are given by
$F=\begin{bmatrix}1&1\\ 0&0\end{bmatrix}$ and $G =\begin{bmatrix}0.1\\ 1\end{bmatrix}$
while \startmodif the \stopmodif matrices $\{C_l\}_{l=1}^2$ and $\{D_m\}_{m=1}^2$ \startmodif \stopmodif have elements defined by \eqref{eq:matrix C:elements} and \eqref{eq:matrix D:elements}. The matrices that are not null are given by $C_1 =\begin{bmatrix} 0.3&0\\0&0\end{bmatrix},
C_2 =\begin{bmatrix}-0.3&0\\0&0\end{bmatrix},
D_1 =\begin{bmatrix} 0.1\\ 1\end{bmatrix},\text{\ and\ }
D_2 =\begin{bmatrix}-0.3\\1\end{bmatrix}$. Applying Proposition \ref{prop:local feedback} and using SeDuMi 1.3, we get $P =\begin{bmatrix}16.1210&7.8345\\7.8345&4.9138\end{bmatrix}$ and $K =\begin{bmatrix}-11.2361&-6.6087\end{bmatrix}$.

Figure \ref{fig:coA:0.1:b} shows some solutions of system \eqref{eq:example:1} in closed loop with the feedback law $\varphi_\ell$ \on, the inclusions $\A\subset\overline{\Omega_1(V_1^\ell)}$ and $\overline{\Omega_1(V_1^\ell)}\subset\V_\mu$. From Proposition \ref{prop:local feedback}, Assumptions \ref{hyp:local stability} and \ref{hyp:inclusion} hold with $c_\ell=1$.

\subsection{Construction of the hybrid feedback law}

Since Assumption \ref{hyp:bounds} holds, we get from the proof of Proposition \ref{prop:global backstepping} that the feedback law given by $\varphi_g(x_1, x_2)=-(1+K_1+2\theta x_1)(x_1+\theta x_1^2+x_2)-x_1/(2K_V)-(x_1-\psi_1(x_1))[c+c\Delta(x_1,x_2)^2/4]/K_V$, where $\Delta(x_1,x_2)= |x_1|\theta(1+|x_1|)+K_V\theta(1+|x_1|)(1+|1+K_1+2\theta x_1|)$, $a=0.01$, $c=10$ and $K_V=1.6286\times10^{3}$, the set $\A+a\mathbf{B}_1$ contains a set that is globally asymptotically stable for $\dx=f_h(x,\varphi_g(x))$. Now applying Theorem \ref{thm:hybrid feedback}, we may design a hybrid feedback law $\mathds{K}$. Let $\tilde{c}_\ell=0.75$ the hybrid controller $\mathds{K}$ defined by \eqref{def:C}-\eqref{def:phi1} in Theorem \ref{thm:hybrid feedback} is such that the origin is globally asymptotically stable for \eqref{eq:example:1} in closed loop.

Consider a simulation with initial condition $(x_1,x_2,q)=(2,0,1)$. 
\begin{figure}[!t]
\centering
\includegraphics[width=7cm]{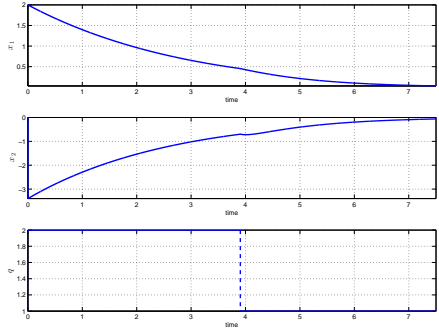}
\caption{Time evolution of a solution of \eqref{eq:example:1} in closed loop with  $\mathbb{K}$ starting from $(2,0,1)$.}
\label{Simu}
\end{figure}
\startmodif Figure \ref{Simu} shows the time evolution of the $x_1$, $x_2$ and $q$ components\footnote{\startmodif Regarding $q$, here it is shown only its first component, because the second one does not change.} of the solution of  \eqref{eq:example:1} in closed loop with $\mathds{K}$. Firstly, \eqref{eq:example:1} is in closed loop with $\varphi_g$ (for $t\in[0,3.9]$), then \eqref{eq:example:1} is in closed loop with $\varphi_\ell$, and the solution converges to the origin.\stopmodif

\begin{comment}
\begin{remark}
In the case in which $\A$ is not included in the basin of attraction of $\dx=f_h(x,\varphi_\ell)$ it may not be possible to design a hybrid feedback law that renders the origin globally asymptotically stable for the closed-loop system, since once system converges to $\A$, it is not ensured that it will converge to the origin.

\begin{minipage}[b]{0.5\linewidth}
   For an illustration, let us consider system \eqref{eq:example:1} with $\theta=1$ and a local stabilizing feedback law given by $\varphi_\ell^\ast(x_1,x_2)=-8x_1+2x_2$.  The Figure at right shows set $\A$ and the basin of attraction of system $\dx=f(x,\varphi_\ell^\ast)$. Thus with such a value of $\theta$ and this local controller, Assumptions \ref{hyp:local stability} and \ref{hyp:inclusion} do not hold.
\end{minipage}
\hfill
\begin{minipage}[b]{0.35\linewidth}
\includegraphics[height=8\baselineskip]{./imgs/sim-theta_1.eps}
\end{minipage}
\end{remark}
\end{comment}

\section{Proofs}\label{sec:proofs}

\subsection{Proof of Proposition \ref{prop:global backstepping}}\label{proof:prop:global backstepping}
\startmodif
In order to prove Proposition \ref{prop:global backstepping}, the  following lemma is required:
\begin{lemma}\label{lem:omega included in A+aB}
	There exist positive constants $a'$ and $K_V$ and a function
	
	\begin{equation}\label{eq:Lyapunov function:V}
		\begin{array}{rcl}
			V:\R^{n-1}\times\R&\to&\R\\
			(x_1,x_2)&\mapsto&V_1(x_1)+K_V(x_2-\varphi_1(x_1))^2
		\end{array}
	\end{equation}
	such that the set $\Omega_{a'}(V)$ satisfies the inclusion
	\begin{equation}\label{inclus}
		\Omega_{a'}(V)\subset\mathbf{A}+a\mathbf{B}_1.
	\end{equation}
\end{lemma}
\begin{proof}
	Consider the sequence of functions $V_k(x_1,x_2)=V_1(x_1)+k(x_1-\varphi_1(x_1))^2$ and of values $a'_k=1/k$. To prove this lemma by contradiction, assume that, $\forall k>0$, inclusion \eqref{inclus} does not hold. In this case, for each $k>0$, there exists a sequence $(x_{1,k},x_{2,k})$ such that $V_k(x_{1,k},x_{2,k})\leq a'_k$ and $(x_{1,k},x_{2,k})\notin\mathbf{A}+a\mathbf{B}_1$.
	Note that, we have
	\begin{equation}\label{eq:proposition 1: V_1<2M x_2-phi_1<2M}
		\left\{\begin{array}{rcccl}
			V_1(x_{1,k})&<&M+\frac{1}{k}&<&2M\\ 
			(x_{2,k}-\varphi_1(x_{1,k}))^2&<&\frac{M}{k}+\frac{1}{k^2}&<&2M.
		\end{array}\right.
	\end{equation}
	The function $V_1$, being proper, yields the sequence $(x_{1,k},x_{2,k})_{k\in\mathbb{N}}$ belongs to a compact subset. Hence, we can  extract a converging subsequence $(x_{1,l},x_{2,l})$ with
$\lim_{l\to\infty}(x_{1,l},x_{2,l})=(x_1^\ast,x_2^\ast)$. We have, with \eqref{eq:proposition 1: V_1<2M x_2-phi_1<2M}, $V_1(x_1^\ast)\leq M$ and $x_2^\ast-\varphi_1(x_1^\ast)=0$. Hence $(x_1^\ast,x_2^\ast)\in\mathbf{A}$. This contradicts the fact that $(x_{1,l},x_{2,l})\notin\mathbf{A}+a\mathbf{B}_1$. Consequently, $\exists a'$ and $K_V$ such that \eqref{inclus} holds.
\end{proof}
	
We are now able to prove Proposition \ref{prop:global backstepping}.
\stopmodif

\begin{proof}
Let $a\in\Rs$ be a constant. We \startmodif will \stopmodif show that there exists a continuous \startmodif controller \stopmodif $\varphi_g$ rendering $\A+a\B_1$ globally asymptotically stable \startmodif for $\dx=f_h(x,\varphi_g(x))$\stopmodif.

\startmodif Define \stopmodif $r_1(x_1,x_2, u)\startmodif:\stopmodif=f_1(x_1,x_2)+h_1(x_1,x_2,u)$. From items a) and b) of Assumption \ref{hyp:bounds} \startmodif we get,\stopmodif\ $\startmodif\forall\stopmodif(x_1,x_2,u)\in\R^{n-1}\times\R\times\R$, 
\begin{equation}\label{eq:dotV1}
\begin{split}
L_{r_1} V_1(x_1,x_2,u) \leq \varepsilon[ \alpha(M)-\alpha(V_1(x_1))]\\
+ L_{r_1} V_1(x_1,x_2,u)- L_{r_1} V_1(x_1&,\psi_1(x_1), u).
\end{split}
\end{equation}

\startmodif Defining, 
$$\forall s\in[0,1],\quad\eta_{x_1,x_2}(s)=sx_2+(1-s)\psi_1(x_1),$$ \startmodif we have \stopmodif 
$$\partial_s r_1(x_1,\eta_{x_1,x_2}(s),u)=\partial_{x_2} r_1 (x_1,\eta_{x_1,x_2}(s),u) (x_2-\psi_1(x_1)).$$ \startmodif This \stopmodif implies \startmodif that
$$r_1(x_1,x_2,u)-r_1(x_1,\psi_1(x_1),u)=(x_2-\psi_1(x_1))\cdot\int_0^1\partial_{x_2} r_1(x_1,\eta_{x_1,x_2}(s),u)\,ds.$$
\stopmodif
Hence, Equation (\ref{eq:dotV1}) becomes 
$$\begin{array}{rcl}
L_{r_1} V_1(x_1,x_2,u) &\leq& \varepsilon[ \alpha(M)-\alpha(V_1(x_1))]\\
&&+ \partial_{x_1} V_1(x_1)(x_2-\psi_1(x_1))\cdot\displaystyle\int_0^1\partial_{x_2} r_1(x_1,\eta_{x_1,x_2}(s),u)\,ds.
\end{array}
$$

Let $\tilde{\psi}$ be the feedback law defined, $\forall(x_1,x_2,\bar{u})\in\R^{n-1}\times\R\times\R$, by 
 
\begin{equation}\label{eq:proof:proposition:psi feedback}
\begin{array}{rcl}
\tilde\psi(x_1, x_2, \bar{u}) = \textstyle\frac{1}{f_2(x_1, x_2)}\left[\frac{\bar{u}}{K_V}+ L_{f_1} \psi_1(x_1, x_2)\right. \\
\left.-\frac{1}{K_V}\partial_{x_1}V_1(x_1) \cdot\displaystyle\int_0^1\partial_{x_2} f_1(x_1,\eta_{x_1,x_2}(s))\,ds\right],\end{array}
\end{equation}
where $K_V$ is given by Lemma \ref{lem:omega included in A+aB}. Letting $u=\tilde\psi(x_1, x_2, \bar{u})$, Equation \eqref{eq:Lyapunov function:V} implies that
$$\begin{array}{rcl}
	L_{f_h}V(x_1,x_2,\tilde\psi(x_1,x_2,\bar u))& \leq &\varepsilon[ \alpha(M)-\alpha(V_1(x_1))]\\
	&&+ (x_2-\psi_1(x_1)) [\bar u + \Upsilon(x_1, x_2,\tilde\psi(x_1,x_2,\bar u))],
\end{array}
$$ with 
$$\begin{array}{rcl}
\Upsilon(x_1, x_2,\tilde\psi(x_1,x_2,\bar u)) &=& \partial_{x_1} V_1(x_1) \cdot\displaystyle\int_0^1\partial_{x_2} h_1(x_1,\eta_{x_1,x_2}(s),\tilde\psi(x_1,x_2,\bar u))\,ds\\
&&+K_V h_2(x_1, x_2, \tilde\psi(x_1,x_2,\bar u))\\
&&- K_V  L_{h_1}\psi_1(x_1, x_2,\tilde\psi(x_1,x_2,\bar u)).
\end{array}
$$
\startmodif From \stopmodif Items b)-d) of Assumption \ref{hyp:bounds}, $\Upsilon$ satisfies $|\Upsilon(x_1, x_2,\tilde\psi(x_1,x_2,\bar u))| \leq \Delta(x_1, x_2)$, \startmodif where\stopmodif  

\begin{equation}\label{eq:Delta}
\begin{split}
\Delta(x_1, x_2) = ||\partial_{ x_1}V_1(x_1)||\displaystyle\int_0^1 \Psi(x_1,\eta_{x_1,x_2}(s))\,ds\\
+K_V\Psi(x_1,x_2)(1+ ||\partial_{ x_1}\psi_1(x_1&)||).
\end{split}
\end{equation}

\startmodif From \stopmodif Cauchy-Schwartz inequality we get, for each positive constant $c$ and $\forall (x_1,x_2,\bar u)\in\R^{n-1}\times \R\times \R$, $$(x_2-\psi_1(x_1))\Upsilon(x_1, x_2,\bar u) \leq \frac{1}{c}+\frac{c}{4} (x_2-\psi_1(x_1))^2\Delta(x_1,x_2)^2.$$ \startmodif Taking\stopmodif

\begin{equation}\label{eq:tildeu}
\bar u=\tilde u := -(x_2-\psi_1(x_1))\left[c+\textstyle\frac{c}{4} \Delta(x_1,x_2)^2\right]
\end{equation}

\noindent it yields, for all $(x_1, x_2)\in\R^{n-1}\times\R$ and $c\geq1$,

\begin{equation}\label{24/10}
\begin{split}
L_{f_h}V(x_1,x_2,\tilde\psi(x_1,x_2)) \leq \varepsilon[ \alpha(M)-\alpha(V_1(x_1))]\\
+\tfrac{1}{c} - c(x_2-\psi_1(x_1&))^2,
\end{split}
\end{equation}
\noindent where, in order to simplify the presentation, we denoted $\tilde\psi(x_1,x_2,\tilde u)$ by $\tilde\psi(x_1,x_2)$. 

Since $V_1$ is proper, the set $$\A_{\geq0} = \left\{\startmodif(x_1,x_2)\in\R^{n-1}\times\R^n\stopmodif:\varepsilon\alpha(V_1(x_1))+(x_2-\psi_1(x_1))^2\leq \varepsilon \alpha(M)+1\right\},$$ compact. \startmodif Let $\zeta = \max_{(x_1,x_2) \in \A_{\geq0}}\{V(x_1,x_2)\}$, for all $c>1$ and $(x_1,x_2)\in\R^n\setminus\overline{\Omega_\zeta(V)}$, we get $L_{f_h}V(x_1,x_2,\psi(x_1,x_2))<0$. In other words, $\overline{\Omega_\zeta(V)}$ is globally asymptotically stable for $\dx=f_h(x,\psi(x))$.\stopmodif

\startmodif Let $K_\alpha>0$ be the Lipschitz constant \stopmodif of $\alpha$ in the compact set $[0,\zeta]$. Hence, 
$$\forall(x_1,x_2)\in\overline{\Omega_\zeta(V)}, \left|\alpha(V_1(x_1))-\alpha(V(x_1,x_2))\right|\leq \frac{K_V K_\alpha}{2} (x_2-\psi_1(x_1))^2.$$

From (\ref{24/10}) \startmodif we get, for all $c>1$ and \stopmodif $(x_1,x_2)\in \overline{\Omega_\zeta(V)}$, $$\begin{array}{rcl}
L_{f_h}V(x_1,x_2,\tilde\psi(x_1,x_2))&\leq&\varepsilon[ \alpha(M)-\alpha(V(x_1,x_2))]+\frac{1}{c}\\
&& - \left(c-\varepsilon\frac{K_V K_\alpha}{2}\right)(x_2-\psi_1(x_1))^2.
\end{array}$$

\startmodif Consider $a'$ given by Lemma \ref{lem:omega included in A+aB} and let \stopmodif $$c_g=\max\left \{\frac{1}{\varepsilon[\alpha(\startmodif a'\stopmodif)-\alpha(M)]}, \varepsilon\frac{K_V K_\alpha}{2},1\right\}$$ it gives, \startmodif for all $c>c_g$ and \stopmodif $(x_1,x_2)\in\overline{\Omega_\zeta(V)}$, $$L_{f_h}V(x_1,x_2,\tilde\psi(x_1,x_2)) \leq \varepsilon\left[ \alpha(\startmodif a'\stopmodif)-\alpha(V(x_1,x_2))\right].$$ Thus, with $c>c_g$, \startmodif $\forall(x_1,x_2)\in\R^n\setminus\overline{\Omega_{a'}(V)}$, we get $L_{f_h}V(x_1,x_2,$ $\tilde\psi(x_1,x_2))<0$.  Therefore, the set $\Omega_{a'}(V)$ is an attractor for \startmodif $\dx=f_h(x,\tilde\psi(x))$\stopmodif. \startmodif Together \stopmodif with (\ref{inclus}), we conclude that $\A+a\B_1$ contains $\Omega_{a'}(V)$ which is globally asymptotically stable with the control $\varphi_g(x)=\tilde\psi(x)$ and $c>c_g$ \startmodif given by \eqref{eq:proof:proposition:psi feedback} and \eqref{eq:tildeu}, that is $$\begin{array}{rcl}
\varphi_g(x_1,x_2)&=&\textstyle\frac{1}{K_V f_2(x_1,x_2)}[K_VL_{f_1}\psi_1(x_1,x_2)-\\
&&\partial_{x_1}V_1(x_1)\cdot\displaystyle\int_0^1\partial_{x_2}f_1(x_1,\eta_{x_1,x_2}(s))\,ds\\
&&-(x_2-\psi_1(x_1))\cdot(c+\textstyle\frac{c}{4}\Delta^2(x_1,x_2))].
\end{array}
$$\stopmodif This concludes the proof of Proposition \ref{prop:global backstepping}.\end{proof}

\subsection{Proof of Theorem \ref{thm:hybrid feedback}}\label{sec:proof:thm 1}

\begin{proof} 
\on Under Assumption \ref{hyp:bounds}, Proposition \ref{prop:global backstepping} can be applied. This allows one to consider a constant $a$ and \off also choose two constant values\footnote{\startmodif Such values exist since Assumptions \ref{hyp:bounds} and \ref{hyp:inclusion} hold, and since $V_{\hat{q}}^\ell$ is a proper function.
\stopmodif} $0<\tilde{c}_\ell<c_\ell$ such that, \startmodif for each $\hat{q}\in\hat{Q}$,
\startmodif
\begin{equation}
\label{choiceof tilde c}
\max _{x\in \A+a\B_1} V_{\hat{q}}^\ell (x)< \tilde{c}_\ell\ .
\end{equation}
\stopmodif
Let us consider the controller $\varphi_g$ given by the proof of Proposition \ref{prop:global backstepping} and \startmodif use it to \stopmodif design a hybrid feedback law $\mathds{K}$ building an hysteresis of \startmodif local controllers $\varphi_{\hat{q}}^\ell$\stopmodif\ and $\varphi_g$ on appropriate domains (see also \cite[Page 51]{4806347} or \cite{Prieur2001} for similar concepts applied to different control problems). Define $Q=\{1,2\}\startmodif\times\hat{Q}\stopmodif$. Consider the subsets \eqref{def:C} and the functions defined in \eqref{def:phi1}.\startmodif\ The state of system \eqref{eq:general system} in closed loop with $\mathds{K}$ is\stopmodif\ $(x,q)\in\R^n\times Q$. 

\startmodif
\textsc{Case 1.} Assume that $q=(2,\hat{q})$. 
	
			i.\ If $x\in\mathscr{C}_{2,\hat{q}}$. Then from \eqref{def:phi1}, we have $\varphi_{2,\hat{q}}(x)=\varphi_{g}(x)$. From Assumptions \ref{hyp:bounds} and Proposition \ref{prop:global backstepping}, $\A$ is globally practically asymptotically stable for $\dx=f_{h}(x,\varphi_{g}(x))$ and $\mathbf{A}\subset\mathscr{D}_{2,\hat{q}}$. Moreover, the solution will not jump until the $x$ component enters in $\mathscr{D}_{2,\hat{q}}$;
			
			ii.\ If $x\in\mathscr{D}_{2,\hat{q}}$. Then from \eqref{def:phi1}, we have $g_{2,\hat{q}}(x)=\{(1,\hat{q})\}$ and, after the jump, the local hybrid controller is selected. Since the value of $x$ does not change during a jump and $\mathscr{D}_{2,\hat{q}}\subset\Omega_{c_\ell}(V_{\hat{q}}^\ell)$, it follows from Assumption \ref{hyp:local stability}, that the origin is locally asymptotically stable for \eqref{eq:local hybrid system}.
	
\textsc{Case 2.} Assume that $q=(1,\hat{q})$.

			i.\ If $x\in\mathscr{D}_{1,\hat{q}}$. Then from \eqref{def:C} and \eqref{def:phi1}, we have either
				
				\underline{i.a.}\ $q^+=(2,\hat{q})$ and, after the jump, the global controller $\varphi_g$ is selected. Since before this jump, the $x$ component must be inside $\overline{\R^n\setminus\Omega_{{c}_\ell}(V_{\hat{q}}^\ell)}$, $\overline{\R^n\setminus\Omega_{{c}_\ell}(V_{\hat{q}}^\ell)}\subset\mathscr{C}_{2,\hat{q}}$ and the $x$ component does change after the jump, the solution follows the behavior prescribed by Case 1.i., after the jump; or \underline{i.b.}\ $q^+=(1,g_{\hat{q}}^\ell(x))$ and, after the jump, a local controller is selected. Since before this jump, the $x$ component must be inside $x\in\overline{\Omega_{c_\ell}(V_{\hat{q}}^\ell)}\cap\mathscr{D}_{\hat{q}}^\ell$ and the $x$ component does change after the jump, it follows from Assumption \ref{hyp:local stability} that the solution  converges to the origin;

			ii.\ If $x\in\mathscr{C}_{1,\hat{q}}$. Then from \eqref{def:phi1}, we have $\varphi_{1,\hat{q}}(x)=\varphi_{\hat{q}}^{\ell}(x)$. From Assumption \ref{hyp:local stability}, the origin is locally asymptotically stable for \eqref{eq:local hybrid system}. Thus, the origin is locally stable and globally attractive. This concludes the proof.
 \end{proof}

\subsection{Proof of Proposition \ref{prop:local feedback}}\label{sec:proof:local feedback}

\begin{proof}
Equation \eqref{eq:LMI:a} may be rewritten in terms of a Linear Matrix Inequality in the matrix variables \eqref{eq:hybrid system} and $W$ given by $W(F+C_l)^T+H(G+D_m)^T+(F+C_l)W+(G+D_m)H^T<0.$
Multiplying this equation at left by and right by a symmetric positive definite matrix $P$ we get, $\forall l\in\L$, $\forall m\in\Mi$, and $\forall x\in\V_\mu$, $x\neq0$,  $x^T(F+C_l+(G+D_m)K)^TPx+x^TP(F+C_l+(G+D_m)K)x<0$.

Equation \eqref{eq:condition:derivee de Lyap} is equivalent to $\mu_s^2W-We_se_s^TW^T>0$, $\forall s=1,2,\ldots,n.$ For each $s=1,2,\ldots,n$, this inequality implies that $x^Te_se_s^Tx<\mu_s^2x^TPx\leq\mu_s^2, \forall x\in\overline{\Omega_1(V_{\hat{q}}^\ell)}$. Since $e_s\cdot x=x_s$, we get $x_s^2<\mu_s^2$, $\forall x\in \overline{\Omega_1(V_{\hat{q}}^\ell)}$, $s=1,2,\ldots,n$.
In other words, $\overline{\Omega_1(V_{\hat{q}}^\ell)}\subset \V_\mu$. Equation \eqref{eq:condition:coA} implies that $x_p^TPx_p\leq1$, $\forall p\in \P$ and thus $\co(\{x_p\}_{p\in\P})\subset \overline{\Omega_1(x^TPx)}$. By Schur complement, \eqref{eq:condition:K} is equivalent to $\mu_u^2W-HH^T\geq0$. This implies $\mu_u^2W^{-T}\geq K^TK$. Then, $\forall x\in\overline{\Omega_1(V_{\hat{q}}^\ell)}$, $x^TK^TKx\leq\mu_u^2x^TPx\leq\mu_u^2$. It concludes the proof.
\end{proof}

\section{Conclusion}\label{sec:conclusion}
A design \on of hybrid feedback laws \off method has been \on presented \off in this paper to combine a backstepping controller with a local feedback law. \on This allows \off us to define a stabilizing control law for nonlinear control systems for which \on the backstepping design procedure \off can not be applied \startmodif to globally stabilize the origin\stopmodif. We have also developed \startmodif conditions to check the assumptions needed for the presented results\stopmodif. In a future work, the authors intend to use these techniques for other classes of nonlinear systems (e.g., cascade systems or in forwarding form).

\bibliographystyle{IEEEtran}
\bibliography{bibliography}

\end{document}